\documentclass{amsart}
\usepackage{amsmath}
\usepackage{amssymb}
\usepackage{amsopn}
\usepackage{epsfig}
\usepackage{amsfonts}
\usepackage{latexsym}
\usepackage{graphicx}
\usepackage{pdfsync}

\theoremstyle{plain}
\newtheorem{theorem}{Theorem}[section]

\newtheorem{proposition}{Proposition}[section]

\newtheorem{example}{Example}[section]

\title[Plumes in kinetic transport]{Plumes in kinetic transport: how the simple random walk can be too simple }

\author[M. Dekking]{Michel Dekking}
\author[D. Kong]{Derong Kong}
\author[A. van Opbroek]{Annegreet van Opbroek}

\address{3TU Applied Mathematics Institute and Delft
University of Technology, Faculty EWI, P.O.~Box 5031, 2600 GA
Delft, The Netherlands.}

\curraddr{Annegreet van Opbroek is currently with the Biomedical Imaging Group at Erasmus Medical Center, Rotterdam, The Netherlands}

\email{F.M.Dekking@tudelft.nl, \quad D.Kong@tudelft.nl, \quad a.vanopbroek@erasmusmc.nl}

\thanks{The second author is partially supported  by the National
Natural Science Foundation of China 10971069 and Shanghai Education Committee Project 11ZZ41.
}

\date{\today}

\newcommand{\ab}{{\textsc{A}}}          
\newcommand{\f}{{\textsc{F}}}          

\newcommand{\expec}[1]{\ensuremath{{\rm E}\!\left[#1\right]}}
\newcommand{\In}{\nu}       
\newcommand{\st}{\pi}    
\newcommand{\prob}[1]{\ensuremath{{\rm P}\left( #1 \right)}}
\newcommand{\proba}[1]{\ensuremath{{\mathbb P}\left( #1 \right)}}

\newcommand{\me}{\mathrm{e}}
\newcommand{\md}{\mathrm{d}}

\newcommand{\cov}[2]{\ensuremath{{\rm Var}\left( #1 ~|~ #2 \right)}}
\newcommand{\cexpec}[2]{\ensuremath{{\rm E}\left( #1 ~|~ #2 \right)}}
\renewcommand{\S}{\mathcal{S}}
\begin{document}
\maketitle

\begin{abstract}
We consider a discrete time  particle model for kinetic transport on the two dimensional integer lattice.
The particle can move due to advection in the $x$-direction and due to dispersion.
This happens when the particle is free, but it can also be adsorbed and then does not move.
When the dispersion of the  particle is modeled by simple random walk, strange phenomena occur.
In the second half of the paper, we resolve these problems and give expressions for the shape of the plume consisting of many particles.

\smallskip


\medskip


 \end{abstract}

\section{Introduction}
We consider a discrete time two dimensional stochastic kinetic transport model, which is described by a two dimensional random walk with a drift
in the $x$ direction. Here the kinetics refers to the possibility that the particle may be adsorbed and released at random times, slowing down the
movement of the particle. Let $S(n)=(S_X(n),S_Y(n))$ be the position of the particle at time $n=0,1,\dots$. In the engineering literature
(see e.g.~\cite{Michalak_Kitanidis_2000}) simulations are performed with a large number of particles independently performing such a slowed down random walk.
This results in a plume of particles, and there is interest in the shape of this plume. One of the remarkable phenomena is that level sets
of the concentration function of the plume need not be convex (see Figure \ref{Fig: 1}, last column).

A natural quantitative way to measure the lateral spread of the plume at position $x$ at time $n$ is to
consider $\sigma^2_x=\cov{S_Y(n)}{S_X(n)=x}$.
The goal of our paper is to compute this conditional variance, and to discuss its properties. Intuitively, $\sigma^2_x$ should be increasing in $x$.
However, if we model the random walk with the natural choice: the simple symmetric random walk on the 2D integer lattice, then (see Figure \ref{Fig: 1}, bottom row) this is not the case,
as we prove in Section \ref{sec: conditional variances}.
A second strange phenomenon that we observe here is that there is symmetry (i.e.,~$\sigma^2_x=\sigma^2_{2n-x}$) when the total number of times
the particle has been free follows a binomial distribution (see Figure \ref{Fig: 1}, bottom left). The background for this is a symmetry property of the asymmetric 2D simple random walk,
as we show in Section~\ref{sec: a symmetry property of assymetric random walk}.

In Section \ref{sec: new types stochastic model} we get rid of the strange behaviour caused by the  simple random walk by tilting the steps. We also show that this choice is `right', as we obtain the same behaviour in the case of Gaussian distributions for the random walk.

For related work lying at the basis of the present paper, see \cite{Uffink_Elfeki_Dekking_Bruining_Kraaikamp_2011}.

\begin{center}
\begin{figure}[h!]
    \includegraphics[width=13cm]{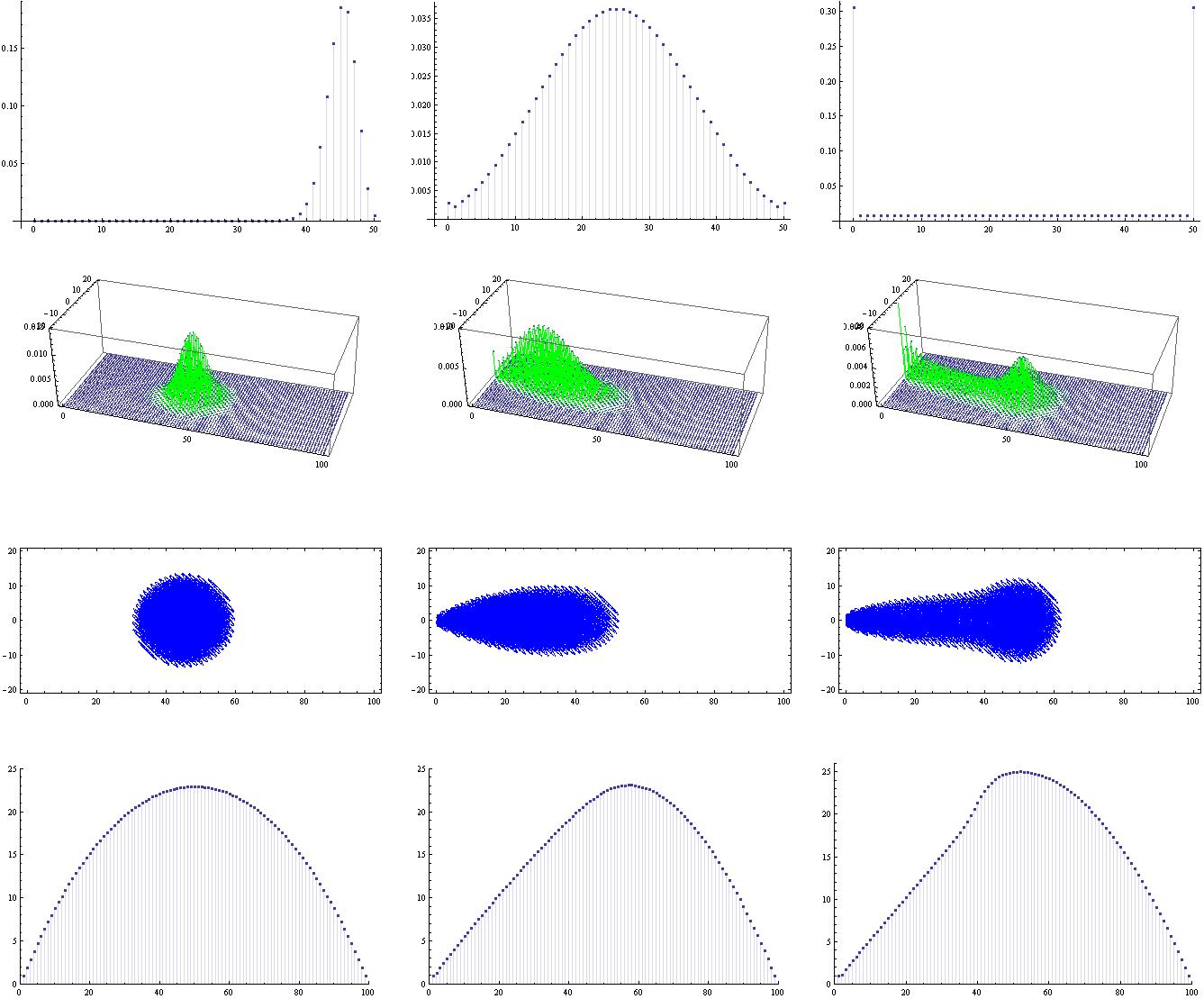}\\
  \caption{ The probability mass function $f_{50}$ of $K_{50}$, the  probability mass function  of $S(50)$, a list contour plot of its mass function and the conditional variance $\cov{S_Y(50)}{S_X(50)=x}$ with $\In=\pi, \alpha=\beta=1/4$.
  In the first column  $a=0.1, b=0.9$; in the second column $a=b=0.1$; in the last column $a=b=0.01$. See Section \ref{sec: stochastic model} and Example \ref{ex: 1} for further explanations.}\label{Fig: 1}
\end{figure}
\end{center}

\section{A discrete time stochastic model for kinetic transport}\label{sec: stochastic model}

 The kinetic transport model is built out of two independent processes: a 2D random walk $\{(X_k, Y_k),\, k\ge 1\}$, and a two-state Markov chain  $\{Z_k,\, k\ge 1\}$ (see  \cite{Dekking_Kong_2011_2}  for the 1D case). We denote by $\rm{P}$ the probability measure on a large enough space to accomodate both processes.
  The two processes describe the behavior of a single particle that in  each unit time interval can
 only be in one of two states: `free' or `adsorbed',
coded by the letters $\f$ and $\ab$. The particle can only
move when it is `free', and in this case its displacement has two
components: dispersion and advection.

Let $(X_k, Y_k)$ be the displacement of the particle due to the dispersion
the $k${th} time that it is `free'. We model the $(X_k,Y_k)$ as independent
identically distributed  random vectors  satisfying
\begin{equation}\label{eq:dispersion X_k Y_k}
\prob{(X_k,Y_k)=(j,0)}=\alpha,\quad\prob{(X_k,Y_k)=(0,j)}=\beta\quad\textrm{for}~j=\pm 1
\end{equation}
where $\alpha,\beta\ge 0,\; \alpha+\beta=1/2$. This means that the particle moves one unit to the left or the right with the same probability $\alpha$, and moves one unit up or down with the same probability $\beta$.
 When the particle is free, the displacement during a unit time interval due to advection  is given by one unit in the $x$ direction.

The kinetics, i.e., the movement from `free' to `adsorbed' and vice versa,  is modeled by  a Markov chain  $\{Z_k, k\ge 1\}$   on the two states $\{\f,\ab\}$ with initial distribution $\In=(\In_\f, \In_\ab)$ and transition matrix
 \begin{equation}\label{eq: transition matrix}
\left[ {\begin{array}{*{20}c}
   {p_{\f,\f} } & {p_{\f,\ab} }  \\
   {p_{\ab,\f} } & {p_{\ab,\ab} }  \\
\end{array}}  \right] =\left[ {\begin{array}{*{20}c}
   1 -a & a  \\
   b & {1-b}  \\
\end{array}} \right].
\end{equation}
Clearly the stationary distribution $\st=(\st_\f,\st_\ab)$ of the Markov chain $\{Z_k\}$ is given by
$\st_\f=b/(a+b)$ and $\st_\ab=a/(a+b).$
Let $K_n$ be the occupation time of the process $\{Z_k\}$  in state $\f$ up to time $n$, i.e.,
  $$K_n=\sum_{k=1}^n {\bf 1}_{\{Z_k=\f\}}.$$
  Then the distribution of $K_n$ is well
known, and is called a Markov binomial distribution (see, e.g., \cite{Dekking_Kong_2011,Omey_Santos_VanGulck_2008}). Let $f_n$ be its distribution, i.e., $f_n(k)=\prob{K_n=k}$ for $0\le k\le n$. Then it is easy to calculate $f_n(k)$ by the following recurrence equation
$$
f_{n+2}(k+1)=(1-b)f_{n+1}(k+1)+(1-a)f_{n+1}(k)-(1-a-b)f_n(k)
$$
with initial conditions
\begin{equation*}
\begin{split}
f_1(0)=\In_\ab,~ f_1(1)=\In_\f,~
 f_2(0)=\In_\ab(1-b), ~ f_2(1)=\In_\ab b+\In_\f a,~ f_2(2)=\In_\f(1-a).
 \end{split}
 \end{equation*}

 Now let $S(n)$ be the position of the particle at time $n$, where $S(0)=(0,0)$. Then by the above
we can write $S(n)$ for $n\ge 1$ as
\begin{equation*}
S(n)=(S_X(n),S_Y(n))=\sum_{k=1}^{K_n} (X_k +1, Y_k).
\end{equation*}

\section{Conditional variance}\label{sec: conditional variances}
In this section we are interested in the conditional variance of $S_Y(n)$ given $S_X(n)=x$.  For a real number $c\in\mathbb{R}$, let $\lfloor c\rfloor$ be the largest integer less than or equal to $c$, and $\lceil c\rceil$ be the least integer greater than or equal to $c$. For two real numbers $c$ and $d$, let $c\vee d:=\max\{c, d\}$ and $c\wedge d:=\min\{c, d\}.$

\begin{proposition}\label{prop: P(S_X,S_Y)}
For $n\ge 1$, and $0\le x\le 2n,  -n\le y\le n$, we have for $x\equiv y\mod{ 2}$
 \begin{equation*}
   \begin{split}
      \prob{S(n)=(x,y)}=\sum_{k=\lceil x/2\rceil}^n f_n(k)\sum_{j=0\vee(x-k)}^{\frac{x+y}{2}\wedge\frac{x-y}{2}} \frac{k!\;\;\alpha^{2j+k-x}\beta^{x-2j}}{j!\,(j+k-x)!\,((x+y)/2-j)!\,((x-y)/2-j)!},
   \end{split}
 \end{equation*}
and $\prob{S(n)=(x, y)}=0$ for $x\not\equiv y\mod 2$.
\end{proposition}

\begin{proof}
For $0\le x\le 2n$ and $-n\le y\le n$,  we consider the probability of the particle being at position $(x,y)$ at time $n$,  conditioned on $K_n=k$. Let $\overline{x}\, (\underline{x},\, \overline{y},\, \underline{y})$ be the number of rightwards (leftwards, upwards, downwards) steps of the random walk $S_k=(X_1+\cdots+X_k+k, Y_1+\cdots+Y_k)$. Suppose the walk makes $\overline{x}=j$ steps to the right. On the event $\{S_k=(x, y)\}$ we then have
\begin{equation*}
\left\{
  \begin{array}{l}
\overline{x}=j\\
\overline{x}-\underline{x}=x-k\\
\overline{y}-\underline{y}=y\\
\overline{x}+\underline{x}+\overline{y}+\underline{y}=k
  \end{array}
\right.
\quad\Longrightarrow\quad
\left\{
  \begin{array}{l}
\overline{x}=j\\
\underline{x}=j+k-x\\
\overline{y}=(x+y)/2-j\\
\underline{y}=(x-y)/2-j.
  \end{array}
\right.
\end{equation*}
Necessarily, $x\equiv y\mod 2$. Since $\overline{x}, \,\underline{x}, \,\overline{y}, \,\underline{y}\ge 0$, it immediately follows that
$$
\max\{0, x-k\}\le j\le \min\Big\{\frac{x+y}{2}, \frac{x-y}{2}\Big\}.
$$
By a standard combinatorial analysis, and using Equation (\ref{eq:dispersion X_k Y_k}) we see that
\begin{equation*}
  \begin{split}
  \prob{S(n)=(x, y)~|~K_n=k}&=\prob{S_k=(x, y)}\\
&=\sum_{0\vee(x-k)}^{\frac{x+y}{2}\wedge\frac{x-y}{2}}\frac{k!\;\;\alpha^{2j+k-x}\beta^{x-2j}}{j!\,(j+k-x)!\,((x+y)/2-j)!\,((x-y)/2-j)!}.
\end{split}
\end{equation*}
Clearly $x\le 2k$, yielding that $k\ge \lceil x/2\rceil$.  Multiplying by $f_n(k)$, and adding over $k$ ranging from $\lceil x/2\rceil$ to $n$, completes the proof of the proposition.
\end{proof}

\begin{proposition}\label{prop: Var(S_Y| S_X)}
  For $n\ge 1$ and $0\le x \le 2n$, the conditional variance  is given by
\begin{equation*}
\begin{split}
\cov{S_Y(n)}{S_X(n)=x}=\frac{\sum\limits_{k=\lceil x/2\rceil}^n f_n(k)\sum\limits_{j=0\vee(x-k)}^{\lfloor x/2\rfloor}(x-2j)\binom{k}{j, j+k-x, x-2j}\alpha^{2j+k-x}(2\beta)^{x-2j}}
  {\sum\limits_{k=\lceil x/2\rceil}^n f_n(k)\sum\limits_{j=0\vee(x-k)}^{\lfloor x/2\rfloor}\binom{k}{j, j+k-x, x-2j}\alpha^{2j+k-x}(2\beta)^{x-2j}}.
\end{split}
\end{equation*}
where $\binom{m}{\ell, s, m-\ell-s}=\frac{m!}{\ell! \,s!\,(m-\ell-s)!}$ denotes a  trinomial coefficient.
\end{proposition}
\begin{proof}
In a similar way as in Proposition $\ref{prop: P(S_X,S_Y)}$, we obtain that
\begin{equation*}
  \begin{split}
    \prob{S_X(n)=x}&=\sum_{k=0}^n f_n(k)\prob{S_X(n)=x~ | ~K_n=k}\\
    &=\sum_{k=\lceil x/2\rceil}^n f_n(k)\sum_{j=0\vee(x-k)}^{\lfloor x/2\rfloor}\frac{k!}{j!\,(j+k-x)!\,(x-2j)!}\alpha^{2j+k-x}(2\beta)^{x-2j}.
  \end{split}
\end{equation*}
By symmetry  $\expec{S_Y(n)~|~S_X(n)=x}=0$. So
\begin{equation*}
  \begin{split}
    \cov{S_Y(n)}{S_X(n)=x}=\sum_{\stackrel{y\equiv x\mod 2}{-n\le y\le n}}y^2\prob{S_Y(n)=y~| ~S_X(n)=x},
  \end{split}
\end{equation*}
and the proposition follows by using Proposition \ref{prop: P(S_X,S_Y)} in the following calculation
\begin{equation*}
  \begin{split}
    &\sum_{\stackrel{y\equiv x\mod 2}{-n\le y\le n}}y^2\prob{S(n)=(x,y)}\\
    =&~\sum\limits_{k=\lceil x/2\rceil}^n f_n(k)\sum\limits_{j=0\vee(x-k)}^{\lfloor x/2\rfloor}\frac{k!}{j!\,(j+k-x)!\,(x-2j)!}\alpha^{2j+k-x}\beta^{x-2j}\sum_{\stackrel{y\equiv x\mod 2}{-n\le y\le n}}y^2\binom{x-2j}{\frac{x-y}{2}-j}\\
    =&~\sum\limits_{k=\lceil x/2\rceil}^n f_n(k)\sum\limits_{j=0\vee(x-k)}^{\lfloor x/2\rfloor}\frac{k!}{j!\,(j+k-x)!\,(x-2j)!}\alpha^{2j+k-x}\beta^{x-2j}\sum_{s=0}^{x-2j}(x-2j-2s)^2\binom{x-2j}{s}\\
    =&~\sum\limits_{k=\lceil x/2\rceil}^n f_n(k)\sum\limits_{j=0\vee(x-k)}^{\lfloor x/2\rfloor}(x-2j)\frac{k!}{j!\,(j+k-x)!\,(x-2j)!}\alpha^{2j+k-x}(2\beta)^{x-2j},
  \end{split}
\end{equation*}
where the first equality follows from the usual convention $\binom{m}{\ell}=0$ for $\ell<0$ or $\ell>m$.
\end{proof}

\begin{example}\label{ex: 1}
Let $n=50, \alpha=\beta=1/4$, and the initial distribution $\In=\st.$ See Figure \ref{Fig: 1}.
\end{example}

 Note that the last column  of Figure \ref{Fig: 1}  shows that two peaks appear in the probability mass function  of $S(n)$ when both $a$ and $b$ are small.
 This double peaking can be  explained (see \cite{Dekking_Kong_2011_2}) by the multimodality of the Markov binomial distribution, i.e., $\{f_n(k)\}_{k=0}^n$
 can be bimodal or even trimodal when both $a$ and $b$ are small (see \cite{Dekking_Kong_2011}).

\section{A  symmetry property of asymmetric random walk}\label{sec: a symmetry property of assymetric random walk}
The first column of Figure \ref{Fig: 1} suggests that when $a+b=1$ the conditional variance is symmetric. This is the content of the next proposition.

\begin{proposition}\label{prop: symmetry 90 case}
  For $n\ge 1$ and $ 0\le x\le n$, if $\In=\st$  and $a+b=1$, we have
  $$
  \cov{S_Y(n)}{S_X(n)=n+x}=\cov{S_Y(n)}{S_X(n)=n-x}.
  $$
\end{proposition}
\begin{proof}
For $\In=\st$ and $a+b=1$, the Markov chain $\{Z_k\}$ is trivial in the sense that $\{Z_k\}$ is independent identically distributed. Then  the behavior of the particle is a  random walk starting at position $(0,0)$ with the following transition probabilities for $k=0,1,\cdots,n-1,$
\begin{equation*}
\begin{split}
 &\prob{S(k+1)=S(k)+(2,0)}=b \alpha,\quad \prob{S(k+1)=S(k)+(0,0)}=b\alpha+a,\\
 &\prob{S(k+1)=S(k)+(1,1)}=\prob{S(k+1)=S(k)+(1,-1)}=b\beta.
 \end{split}
 \end{equation*}
  For convenience we shift the random walk steps by adding $(-1,0)$ to them and let the particle start at position $(n,0)$. Then  the corresponding  transition probabilities
  of the random walk are given by
  \begin{equation*}
    \begin{split}
     & \prob{\hat{S}(k+1)=\hat{S}(k)+(1,0)}=b\alpha,\quad \prob{\hat{S}(k+1)=\hat{S}(k)+(-1,0)}=b\alpha+a,\\
     & \prob{\hat{S}(k+1)=\hat{S}(k)+(0,1)}=\prob{\hat{S}(k+1)=\hat{S}(k)+(0,-1)}=b\beta.
    \end{split}
  \end{equation*}
The result now follows from the following theorem with $\omega=b\alpha, \varepsilon=b\alpha+a$ and $\gamma=\delta=b\beta$.
\end{proof}

\begin{theorem}
Let $(S_n)=((S_n^{(1)}, S_n^{(2)}))$ be a simple random walk in the plane, i.e., $S_0=(0,0)$  and $(S_n)$ has increments $\mathcal{X}_k$ with
\begin{equation*}
     \proba{\mathcal{X}_k=(1,0)}=\omega,\quad \proba{\mathcal{X}_k=(-1,0)}=\varepsilon,\quad
     \proba{\mathcal{X}_k=(0,1)}=\gamma,\quad \proba{\mathcal{X}_k=(0,-1)}=\delta,
  \end{equation*}
where $\omega+\varepsilon+\gamma+\delta=1$. Then for $x=0,1,2,\dots,n$ and $-n\le y\le n$,
$$
\proba{S_n^{(2)}=y~\big|~S_n^{(1)}=x}=\proba{S_n^{(2)}=y~\big|~S_n^{(1)}=-x}.
$$
\end{theorem}
\begin{proof}
A similar combinatorial analysis as in the proof of Proposition \ref{prop: P(S_X,S_Y)} yields that for $x=0,1,\cdots,n$ and $-n\le y\le n$
\begin{equation*}
  \begin{split}
    \proba{S_n^{(1)}=x, S_n^{(2)}=y}&=\sum_{j=0}^{\frac{n-x-y}{2}\wedge\frac{n-x+y}{2}}\frac{ n!\;\;
    \omega^{j+x}\varepsilon^j\gamma^{(n-x-2j+y)/2}\delta^{(n-x-2j-y)/2}}{j!\,(j+x)!\,((n-x-y)/2-j)!\,((n-x+y)/2-j)!},\\
    \proba{S_n^{(1)}=-x, S_n^{(2)}=y}&=\sum_{j=0}^{\frac{n-x-y}{2}\wedge\frac{n-x+y}{2}}\frac{ n!\;\;
    \omega^{j}\varepsilon^{j+x}\gamma^{(n-x-2j+y)/2}\delta^{(n-x-2j-y)/2}}{j!\,(j+x)!\,((n-x-y)/2-j)!\,((n-x+y)/2-j)!}
  \end{split}
\end{equation*}
if $y\equiv n-x \mod 2$, and $\proba{S_n^{(1)}=x, S_n^{(2)}=y}=\proba{S_n^{(1)}=-x, S_n^{(2)}=y}=0$ if $y\not\equiv n-x \mod 2.$
Similarly,
\begin{equation*}
\begin{split}
\proba{S_n^{(1)}=x}&=\sum_{i=0}^{\lfloor\frac{n-x}{2}\rfloor}\frac{n!}{i!\,(i+x)!\,(n-x-2i)!} \omega^{i+x}\varepsilon^i(\gamma+\delta)^{n-2i-x},\\
\proba{S_n^{(1)}=-x}&=\sum_{i=0}^{\lfloor\frac{n-x}{2}\rfloor}\frac{n!}{i!\,(i+x)!\,(n-x-2i)!} \omega^{i}\varepsilon^{i+x}(\gamma+\delta)^{n-2i-x}.
\end{split}
\end{equation*}
Then it is easy to check that for $x=0,1,\cdots,n$ and $-n\le y\le n$
\begin{equation*}
  \begin{split}
    \proba{S_n^{(1)}=x, S_n^{(2)}=y}\proba{S_n^{(1)}=-x}
    =\proba{S_n^{(1)}=-x, S_n^{(2)}=y}\proba{S_n^{(1)}=x},
  \end{split}
\end{equation*}
which finishes the proof of the theorem.
\end{proof}

\section{Forty five degree dispersion}\label{sec: new types stochastic model}
To fix the strange behavior observed when the dispersion is 2D
simple random walk, we consider now a well known variant  with increments
$$
\prob{(X_k,Y_k)=(1,j)}=\alpha,\quad\prob{(X_k,Y_k)=(-1,j)}=\beta\quad\textrm{for}~j=\pm 1,
$$
where $\alpha, \beta\ge 0$ and $\alpha+\beta=1/2.$ Let $(\S(n))$ be the corresponding kinetic transport process.
\begin{proposition}\label{prop: Var_45 degree}
  For $n\ge 1$ and $0\le x\le n, -n\le y\le n$, the distribution of $\S(n)$ is given by
  $$\prob{\S(n)=(2x,y)}=\sum_{k=x}^n f_n(k)\binom{k}{x}\binom{k}{\frac{k+y}{2}}\alpha^x\beta^{k-x},$$
 and the conditional variance is given by
  $$
  \cov{\S_Y(n)}{\S_X(n)=2x}=\frac{\sum_{k=x}^n k  f_n(k)\binom{k}{x}(2\alpha)^x(2\beta)^{k-x}}{\sum_{k=x}^n f_n(k)\binom{k}{x}(2\alpha)^x(2\beta)^{k-x}}.
  $$
 \end{proposition}
 \begin{proof}
The proof of the proposition is quite similar to that for Proposition \ref{prop: P(S_X,S_Y)}.  For $0\le x\le n$ and $-n\le y\le n$, we consider the probability of the particle being at position $(2x, y)$ at time $n$ by conditioning on $K_n=k$.  In order to achieve $\S_X(n)=k+X_1+\cdots+X_k=2x$ the particle needs to move $x$ steps to the right and $k-x$ steps to the left. Also to obtain $\S_Y(n)=Y_1+\cdots Y_k=y$ the particle needs to move $(k+y)/2$ steps up and $(k-y)/2$ steps down. Let $j$ be the number of steps of type $(1,1)$, then the particle moves $x-j$ steps of type $(1,-1)$, $(k+y)/2-j$ steps of type $(-1,1)$, and the remaining $k-(j+x-j+(k+y)/2-j)=(k-y)/2-(x-j)$ steps of type $(-1,-1)$. To ensure the nonnegativity of all of these quantities, we have $x\le k\le n$ and $\max\{0,(y-k)/2+x\}\le j\le \min\{x,(k+y)/2\}$.  We conclude by using Vandermonde's identity, obtaining  that for $0\le x\le n, -n\le y\le n$ and $k\equiv y\mod 2$
$$
\prob{\S(n)=(2x,y)~|~K_n=k}=\sum_{j=0}^{(k+y)/2}\binom{k}{x}\binom{x}{j}
\binom{k-x}{\frac{k+y}{2}-j}\alpha^x\beta^{k-x}=\binom{k}{x}\binom{k}{\frac{k+y}{2}}\alpha^x\beta^{k-x}.
$$
For the conditional variance $\cov{\S_Y(n)}{\S_X(n)=2x}$, we still need the probability of the particle being at position $2x$ at time $n$. In a similar way as above it follows that
$$\prob{\S_X(n)=2x~|~K_n=k}=\binom{k}{x}(2\alpha)^x(2\beta)^{k-x}.$$
By symmetry $\expec{\S_Y(n)~|~\S_X(n)=2x}=0$. So
$$
\cov{\S_Y(n)}{\S_X(n)=2x}=\sum_{y=-n}^n y^2\prob{\S_Y(n)=y~|~\S_X(n)=2x},
$$
and the proposition follows by the following calculation
\begin{equation*}
\begin{split}
 & \sum_{y=-n}^n y^2\prob{\S(n)=(2x,y)}=\sum_{y=-n}^n y^2\sum_{\stackrel{k\equiv y\mod 2}{x\le k\le n}}f_n(k)\binom{k}{x}\binom{k}{\frac{k+y}{2}}\alpha^x\beta^{k-x}\\
  =&~\sum_{k=x}^n f_n(k)\binom{k}{x}\alpha^x\beta^{k-x}\sum_{\stackrel{y\equiv k\mod 2}{-n\le y\le n}}y^2\binom{k}{\frac{k+y}{2}}=\sum_{k=x}^n f_n(k)\binom{k}{x}\alpha^x\beta^{k-x}\sum_{s=0}^k(k-2s)^2\binom{k}{s}\\
  =&~\sum_{k=x}^n k  f_n(k)\binom{k}{x}(2\alpha)^x(2\beta)^{k-x}.
  \end{split}
\end{equation*}
 \end{proof}

 \begin{example}
Let $n=50, \alpha=\beta=1/4$, and the initial distribution $\In=\st$. See Figure \ref{Fig: 2}.
\begin{center}
\begin{figure}[h!]
    \includegraphics[width=13cm]{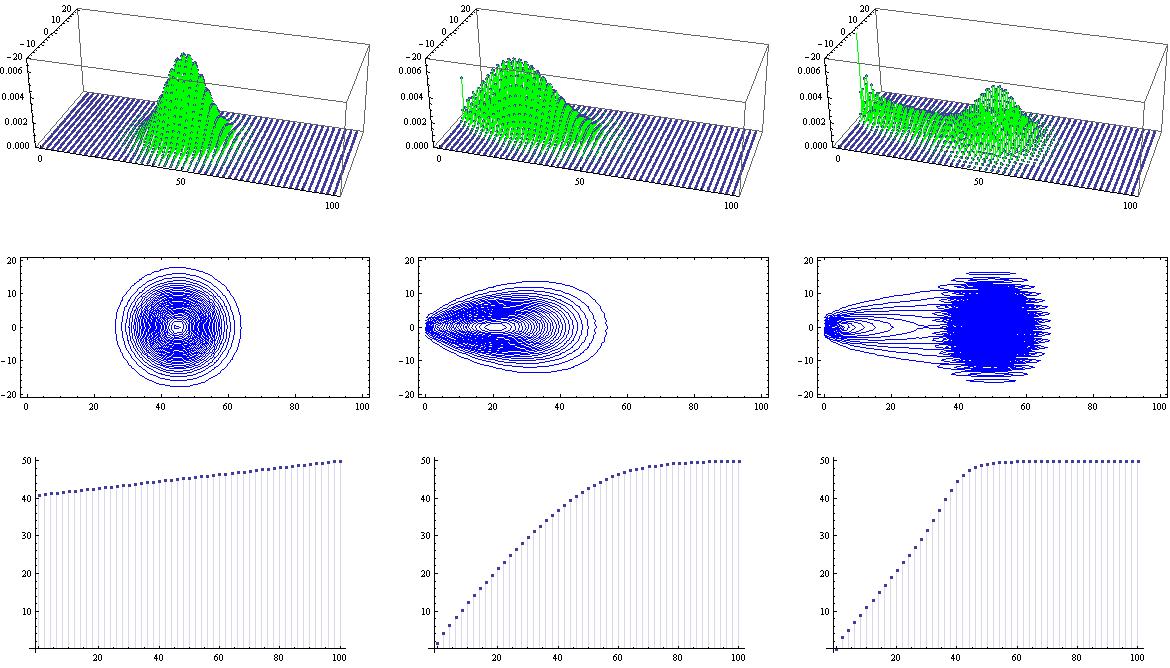}\\
  \caption{The  probability mass function of $\S(50)$, a list contour plot of its mass function and the conditional variance $\cov{\S_Y(50)}{\S_X(50)=2x}$ with $\In=\st, \alpha=\beta=1/4$.  In the first column  $a=0.1, b=0.9$; in the second column $a=b=0.1$; in the last column $a=b=0.01$.}\label{Fig: 2}
\end{figure}
\end{center}
\end{example}
Figure \ref{Fig: 2} suggests that the conditional variance $\cov{\S_Y(n)}{\S_X(n)=2x}$ is increasing in $x$. This is indeed the case, and can be shown by using Proposition \ref{prop: Var_45 degree}.
\begin{proposition}\label{prop: increasing 45 case}
  The conditional variance $\cov{\S_Y(n)}{\S_X(n)=2x}$ is an increasing function in $x$.
\end{proposition}
\begin{proof}
  Let $g(x):=\sum_{k=x}^n k  f_n(k)\binom{k}{x}(2\alpha)^x(2\beta)^{k-x}$ and $h(x):=\sum_{k=x}^n f_n(k)\binom{k}{x}(2\alpha)^x(2\beta)^{k-x}$. According to Proposition \ref{prop: Var_45 degree} $\cov{\S_Y(n)}{\S_X(n)=2x}=g(x)/h(x)$. For $0\le x<n$, note that $\cov{\S_Y(n)}{\S_X(n)=2(x+1)}\ge \cov{\S_Y(n)}{\S_X(n)=2x}$ is equivalent to $g(x+1)h(x)\ge g(x)h(x+1)$, since $h(x)>0$. Consider the difference
  \begin{equation*}
    \begin{split}
      &g(x+1)h(x)-g(x)h(x+1)\\
      =&~\sum_{k=x+1}^n\sum_{\ell=x}^n (k-\ell)\binom{k}{x+1}\binom{\ell}{x} f_n(k)f_n(\ell)(\alpha/\beta)^{2x+1}(2\beta)^{k+\ell}\\
      \ge &~\sum_{k,\ell=x+1}^n (k-\ell)\binom{k}{x+1}\binom{\ell}{x} w(k,\ell),
    \end{split}
  \end{equation*}
  where $w(k,\ell)=f_n(k)f_n(\ell)(\alpha/\beta)^{2x+1}(2\beta)^{k+\ell}\ge 0$.
 Using the symmetry   $w(k,\ell)=w(\ell,k)$,  the proposition then follows from
  \begin{equation*}
\begin{split}
  &g(x+1)h(x)-g(x)h(x+1)\\
  \ge&~\frac{1}{2}\sum_{k,\ell=x+1}^n (k-\ell)\binom{k}{x+1}\binom{\ell}{x} w(k,\ell)+\frac{1}{2}\sum_{k,\ell=x+1}^n (\ell-k)\binom{\ell}{x+1}\binom{k}{x} w(\ell,k)\\
  =&~\frac{1}{2}\sum_{k,\ell=x+1}^n w(k,\ell)(k-\ell)\Big[\binom{k}{x+1}\binom{\ell}{x}-\binom{\ell}{x+1}\binom{k}{x}\Big]\\
  =&~\frac{1}{2(x+1)}\sum_{k,\ell=x+1}^n w(k,\ell)\binom{k}{x}\binom{\ell}{x}(k-\ell)^2~\ge 0.
    \end{split}
  \end{equation*}
\end{proof}

Proposition \ref{prop: increasing 45 case} seems to suggest that the strange phenomenon in the conditional variance (not increasing as in Figure \ref{Fig: 1}) will disappear if we consider independent increments $(X_k, Y_k)$ where  $X_k$ is also independent of $Y_k$ for $k=1,2,\cdots,K_n$. However, we will give a mean zero
nearest neighbor random walk example, where this phenomenon is still present.

Let $(X_k, Y_k), k\ge 1$ be independent identical distributed random vectors with distribution
\begin{equation*}
  \begin{split}
    \prob{(X_k, Y_k)=(1,\pm 1)}=\prob{(X_k, Y_k)=(-1,\pm 1)}=\xi,\quad
    \prob{(X_k, Y_k)=(\pm 1,0)}=\frac{1}{2}-2\xi.
  \end{split}
\end{equation*}
where $ 0<\xi<1/4$. Obviously, $X_k$ is independent of $Y_k$. Let $\bar{S}(n)=(\bar{S}_X(n), \bar{S}_Y(n))$ be the position of the particle at time $n$. One can compute its probability mass function and conditional variance $\cov{\bar{S}_Y(n)}{\bar{S}_X(n)=x}$.
\begin{example}
  Let $n=25, a=b=0.01, \xi=0.2$, and $\In=\st$. See Figure \ref{fig: 4}.
 \begin{center}
\begin{figure}[h!]
    \includegraphics[width=14cm]{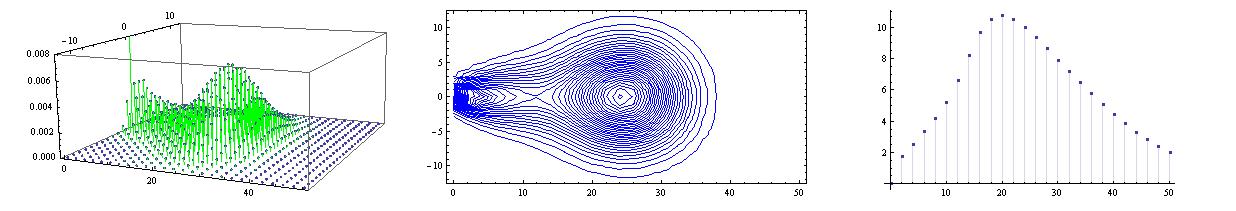}\\
  \caption{The  probability mass function of $\bar{S}(25)$, a list contour plot of its mass function and the conditional variance $\cov{\bar{S}_Y(25)}{\bar{S}_X(25)=2x}$ with $\In=\st, a=b=0.01,\; \xi=0.2$.}\label{fig: 4}
\end{figure}
\end{center}
\end{example}

\section{Gaussian dispersion}

It is natural (and common practice in the engineering literature, see e.g.,~\cite{Michalak_Kitanidis_2000}) to consider a Gaussian  dispersion. Let $\mathbf{S}(n)$ be the position of the particle at time $n$ as described in Section \ref{sec: stochastic model} except that $(X_k,Y_k)$ now has a Gaussian distribution with independent marginals
$$
X_k\stackrel{\mathbf{d}}{=}\mathcal{N}(0,2\alpha),\quad Y_k\stackrel{\mathbf{d}}{=}\mathcal{N}(0,2\beta)\quad\textrm{for}~\alpha,\beta>0.
$$

\begin{proposition}\label{prop: var_normal case}
  For $n\ge 1$, the distribution $\mu_n$ of the random variable $\mathbf{S}(n)$ can be written as
   $$
   \mu_n=f_n(0)\delta_{(0,0)}+(1-f_n(0))\tilde{\mu}_n,
   $$
   where $\delta_{(0,0)}$ is Dirac measure at the point $(0,0)$ and $\tilde{\mu}_n$ is the distribution of  a continuous random variable having density function
  $$\tilde{f}_{\mathbf{S}(n)}(x,y)=\frac{1}{1-f_n(0)}\sum_{k=1}^n \frac{f_n(k)}{4\pi k\sqrt{\alpha \beta}}\exp{\Big(-\frac{(x-k)^2}{4 k \alpha}-\frac{y^2}{4k \beta}\Big)}.$$
Moreover, the conditional variance is given by
  $$\cov{\mathbf{S}_Y(n)}{\mathbf{S}_X(n)=x}=2\beta(1-f_n(0))\frac{\sum_{k=1}^n \exp{\big(-(x-k)^2/(4k \alpha)\big)}f_n(k)\sqrt{k}}{\sum_{k=1}^n \exp{\big(-(x-k)^2/(4k \alpha)\big)}f_n(k)/\sqrt{k}}.$$
\end{proposition}
\begin{proof}
As in \cite{Dekking_Kong_2011_2} one easily finds the characteristic function $\varphi_{\mathbf{S}(n)}$ of $\mathbf{S}(n)$, and as in \cite{Dekking_Kong_2011_2} one shows that $(0,0)$ is the unique atom of $\mu_n$, with mass $f_n(0)$. Then the density function of the continuous part of $\mu_n$ follows by inverse Fourier transformation:
\begin{equation*}
  \begin{split}
    \tilde{f}_{\mathbf{S}(n)}(x,y)&=\frac{1}{1-f_n(0)}\frac{1}{(2\pi)^2}\int\int \big(\varphi_{\mathbf{S}(n)}(u_1,u_2)-f_n(0)\big)\me^{-i(u_1 x+u_2 y)}\md u_1\md u_2\\
    &=\frac{1}{1-f_n(0)}\sum_{k=1}^n f_n(k)\frac{1}{(2\pi)^2}\int\int \me^{-i u_1 x+k(i u_1-\alpha u_1^2)}\me^{-i u_2 y-k\beta u_2^2}\md u_1\md u_2\\
    &=\frac{1}{1-f_n(0)}\sum_{k=1}^n f_n(k)\frac{1}{4 \pi k\sqrt{\alpha \beta}}\exp{\Big(-\frac{(x-k)^2}{4k\alpha}-\frac{y^2}{4 k\beta}\Big)}.
  \end{split}
\end{equation*}
Using symmetry, the conditional variance is given by
\begin{equation*}
  \begin{split}
  \cov{\mathbf{S}_Y(n)}{\mathbf{S}_X(n)=x}=\cexpec{\mathbf{S}_Y^2(n)}{\mathbf{S}_X(n)=x}
  =(1-f_n(0))\int_{-\infty}^{\infty}y^2\frac{\tilde{f}_{\mathbf{S}(n)}(x,y)}{\tilde{f}_{\mathbf{S}_X(n)}(x)}\md y,
  \end{split}
\end{equation*}
where $\tilde{f}_{\mathbf{S}_X(n)}(x)=\int_{-\infty}^{\infty}\tilde{f}_{\mathbf{S}(n)}(x,y)\md y.$
The proof is then finished by a straightforward  calculation.
\end{proof}

 \begin{example}
Let $n=50, \alpha=\beta=1/4$, and the initial distribution $\In=\st$. See Figure \ref{Fig: 3}.
\begin{center}
\begin{figure}[h!]
    \includegraphics[width=13cm]{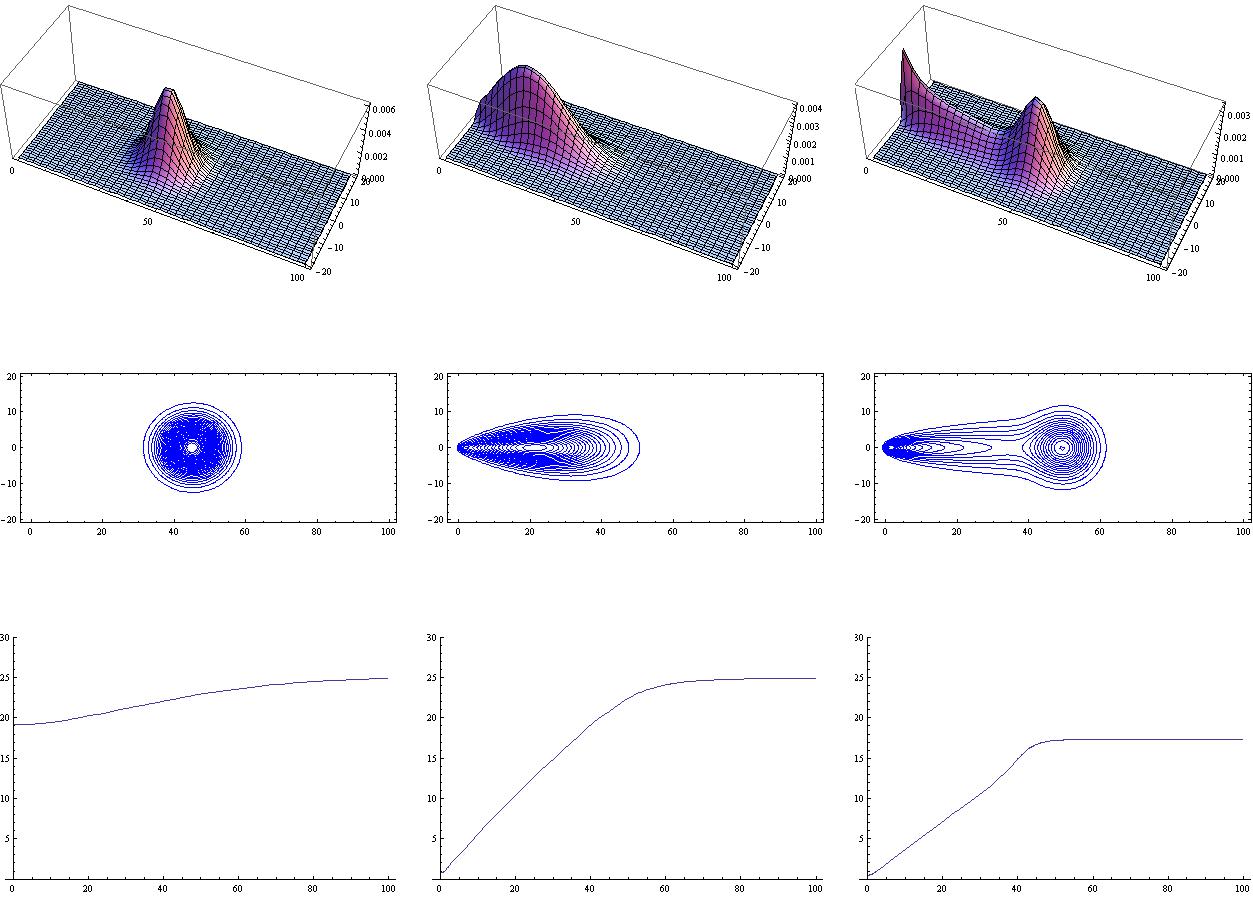}\\
  \caption{The density $\tilde{f}_{\mathbf{S}(50)}$, a contour plot of the density and the conditional variance $\cov{\mathbf{S}_Y(50)}{\mathbf{S}_X(50)=x}$ with $\In=\st, \alpha=\beta=1/4$. In the first column  $a=0.1, b=0.9$; in the second column $a=b=0.1$; in the last column $a=b=0.01$.}\label{Fig: 3}
\end{figure}
\end{center}
\end{example}
Figure \ref{Fig: 3} suggests that the conditional variance $\cov{\mathbf{S}_Y(n)}{\mathbf{S}_X(n)=x}$ is increasing in $x$. This is indeed the case, and is the content of  the following proposition.
\begin{proposition}\label{prop: increasing Gaussian case}
  The conditional variance $\cov{\mathbf{S}_Y(n)}{\mathbf{S}_X(n)=x}$ is an increasing function in $x$.
\end{proposition}
\begin{proof}
Since by Proposition \ref{prop: var_normal case}  $\cov{\mathbf{S}_Y(n)}{\mathbf{S}_X(n)=x}$ is a differentiable function, we only need to show that its first derivative is nonnegative. For simplicity, we write for
$$g(x)=\sum_{k=1}^n f_n(k)\sqrt{k}\me^{-\frac{(x-k)^2}{4k\alpha}},\quad h(x)=\sum_{k=1}^n \frac{f_n(k)}{\sqrt{k}}\me^{-\frac{(x-k)^2}{4k\alpha}}.$$
So $\cov{\mathbf{S}_Y(n)}{\mathbf{S}_X(n)=x}'=\frac{2\beta(1-f_n(0))}{h^2(x)}\big(g'(x)h(x)-g(x)h'(x)\big)$. Note that
\begin{equation*}
  \begin{split}
    &g'(x)h(x)-g(x)h'(x)\\
    =&~\frac{1}{2\alpha}\sum_{k,\ell=1}^n\frac{f_n(k)}{\sqrt{k}}\frac{f_n(\ell)}{\sqrt{\ell}}
    \me^{-\frac{(x-k)^2}{4k\alpha}-\frac{(x-\ell)^2}{4\ell\alpha}}\frac{\ell-k}{k}(x-k)\\
    =:&~\frac{1}{2\alpha}\sum_{k,\ell=1}^n w(k,\ell)\frac{\ell-k}{k}(x-k),
  \end{split}
\end{equation*}
where $w(k,\ell)=\frac{f_n(k)}{\sqrt{k}}\frac{f_n(\ell)}{\sqrt{\ell}}\me^{-\frac{(x-k)^2}{4k\alpha}-\frac{(x-\ell)^2}{4\ell\alpha}}\ge 0$. Using the symmetry   $w(k,\ell)=w(\ell,k)$, the proposition then follows  from
\begin{equation*}
  \begin{split}
    &g'(x)h(x)-g(x)h'(x)\\
    =&~\frac{1}{4\alpha}\Big[\sum_{k,\ell=1}^n w(k,\ell)\frac{\ell-k}{k}(x-k)+\sum_{k,\ell=1}^n w(\ell,k)\frac{k-\ell}{\ell}(x-\ell)\Big]\\
    =&~\frac{1}{4\alpha}\sum_{k,\ell=1}^n w(k,\ell)(\ell-k)\big[(\frac{x}{k}-1)-(\frac{x}{\ell}-1)\big]
    =\frac{x}{4\alpha}\sum_{k,\ell=1}^n w(k,\ell)\frac{(\ell-k)^2}{k\ell}~\ge 0.
  \end{split}
\end{equation*}
\end{proof}

\bibliography{Markov-bin-AKAD}

\begin{thebibliography}{1}

\bibitem{Dekking_Kong_2011}
Michel Dekking and DeRong Kong.
\newblock Multimodality of the {M}arkov binomial distribution.
\newblock {\em Journal of Applied Probability}, 48(4):938--953, 2011.

\bibitem{Dekking_Kong_2011_2}
Michel Dekking and DeRong Kong.
\newblock A simple stochastic kinetic transport model.
\newblock {\em arXiv: 1106.2912v1, to appear in Advances in Applied
  Probability}, 2012.

\bibitem{Michalak_Kitanidis_2000}
A.M. Michalak and Peter~K. Kitanidis.
\newblock Macroscopic behavior and random-walk particle tracking of kinetically
  sorbing solutes.
\newblock {\em Water Resources Research}, 36(8):2133--2146, 2000.

\bibitem{Omey_Santos_VanGulck_2008}
E.~Omey, J.~Santos, and S.~Van~Gulck.
\newblock A {M}arkov-binomial distribution.
\newblock {\em Appl. Anal. Discrete Math.}, 2(1):38--50, 2008.

\bibitem{Uffink_Elfeki_Dekking_Bruining_Kraaikamp_2011}
G.~Uffink, A.~Elfeki, F.M. Dekking, J.~Bruining, and C.~Kraaikamp.
\newblock Understanding the non-gaussianity of reactive transport; from
  particle dynamics to {PDE}'s.
\newblock {\em Transp Porous Med}, pages 1--25, 2011.

\end{thebibliography}
\bibliographystyle{plain}
\end{document}